\documentclass[11pt,reqno]{amsart}

\usepackage{amsmath,amssymb,amsthm,mathtools}
\usepackage[hidelinks]{hyperref}

\newtheorem{theorem}{Theorem}[section]
\newtheorem{proposition}[theorem]{Proposition}
\newtheorem{lemma}[theorem]{Lemma}
\newtheorem{corollary}[theorem]{Corollary}
\theoremstyle{definition}
\newtheorem{definition}[theorem]{Definition}

\theoremstyle{remark}
\newtheorem{remark}[theorem]{Remark}

\newcommand{\Q}{\mathbb{Q}}
\newcommand{\Unit}{[0,1]}
\newcommand{\Nat}{\mathbb{N}}
\newcommand{\UnitInterval}{[0,1]}
\newcommand{\UpSet}{\mathord{\uparrow}}

\title[A Quasicontinuous Domain Without an Interval Retract]
  {A Quasicontinuous Domain Without an Interval Retract}

\author{Chong Shen}
\address{School of Science, Beijing University of Posts and Telecommunications,
	Beijing, China}

\author{Xiaoyong Xi}
\address{School of Mathematics and Statistics,\\
Yancheng Teachers University, Yancheng, Jiangsu, China}

\author{Dongsheng Zhao}
\address{Mathematics and Mathematical Education, National Institute of Education,
Nanyang Technological University, 1 Nanyang Walk, Singapore}

\date{}

\begin{document}

\begin{abstract}
We give negative answers to two questions of X.~Xu concerning the occurrence
of the unit interval in nonquasialgebraic quasicontinuous domains.  We
construct, directly from the binary tree and the binary-value map, a
well-founded quasicontinuous dcpo
\[
P=2^{<\omega}\mathbin{\dot\cup}[0,1]
\]
which is not quasialgebraic. Moreover,   $P$ contains no
sub-dcpo isomorphic to $[0,1]$, and $[0,1]$ is not a Scott-continuous
retract of $P$. 
\end{abstract}

\maketitle

\section{Introduction}

The unit interval $\Unit$ carries its usual order  is a basic
continuous dcpo which is not algebraic.  A recurring theme in domain theory
is to understand failures of algebraicity through the occurrence of such a
canonical interval; standard background may be found in
\cite{AbramskyJung,Gierz,GoubaultLarrecq}.  Quasicontinuity is a natural
weakening of continuity in which approximation is described by finite sets,
rather than by individual elements; it was introduced in this form by
Gierz, Lawson, and Stralka \cite{GLS}.

For continuous domains, Jia, Li, and Luan showed that every nonalgebraic
continuous domain admits $\Unit$ as a Scott-continuous retract
\cite{JiaLiLuan}.

In the quasicontinuous setting, Xu proved that every nonquasialgebraic
quasicontinuous domain admits a surjective monotone Lawson-continuous map
onto $\Unit$ \cite[Theorem~3.1]{Xu}.  Thus the interval is always present
as a monotone image.  The issue is whether it must already occur internally
as a retract or as a sub-dcpo.  Here, a Scott-continuous retract means
Scott-continuous maps
\[
        r:P\longrightarrow\Unit,\qquad e:\Unit\longrightarrow P
\]
such that $r\circ e=\operatorname{id}_{\Unit}$.

The following two questions were raised by Xu \cite{XuQuestions}.

\medskip
\noindent\textbf{Problem 1.14 (Xu).}
If $P$ is a quasicontinuous domain which is not quasialgebraic, must
$\Unit$ be a Scott-continuous retract of $P$?

\medskip
\noindent\textbf{Problem 1.15 (Xu).}
If $P$ is a quasicontinuous domain which is not quasialgebraic, must $P$
contain a sub-dcpo isomorphic to $\Unit$?

\medskip

For arbitrary dcpos, the two formulations are equivalent; see
Proposition~\ref{prop:equivalence}.  We retain both because they arise
naturally from the questions above.

\begin{theorem}[Direct counterexample]\label{thm:counterexample}
There is a well-founded quasicontinuous dcpo $P$ which is neither algebraic
nor quasialgebraic and has the following properties.
\begin{enumerate}
  \item Its maximal elements form an antichain of cardinality
  $2^{\aleph_0}$.
  \item Every chain in $P$ is order-isomorphic to a suborder of $\omega+1$.
  \item $P$ contains no sub-dcpo isomorphic to $\Unit$, and $\Unit$ is not a
  Scott-continuous retract of $P$.
\end{enumerate}
\end{theorem}

Thus both questions have negative answers in general.  The construction is
defined entirely from finite binary words, infinite binary expansions, and
their numerical values under the binary-value map.  Moreover, it admits an
explicit Scott-continuous surjection onto $\Unit$.  Hence the obstruction is
stronger than the failure of a Lawson-continuous image to split: even a
Scott-continuous interval image need not be a Scott-continuous retract.

Section~\ref{sec:preliminaries} recalls the finite-set formulation of
quasicontinuity.  In Section~\ref{sec:construction} we construct the dcpo
and establish its elementary order-theoretic properties.  The following two
sections prove that the example is quasicontinuous and not quasialgebraic.
We then answer Xu's two problems and construct the Scott-continuous
surjection.  Finally, Section~\ref{sec:criterion} records the general
interval criterion used to identify the two formulations of the questions.


\section{Preliminaries}\label{sec:preliminaries}

We use the usual finite-set formulation of quasicontinuity; see
\cite[Chapter III]{Gierz}.  If $P$ is a dcpo and $A,B\subseteq P$, put
\[
 A\ll B
 \quad\Longleftrightarrow\quad
 \text{for every directed }D\subseteq P,\quad
 \bigvee D\in\uparrow B\Longrightarrow D\cap\uparrow A\neq\varnothing.
\]
For a singleton $B=\{x\}$, write $A\ll x$.  As usual,
\[
 \uparrow A=\{p\in P:\text{there is }a\in A\text{ with }a\leq p\}.
\]

\begin{definition}
A dcpo $P$ is \emph{quasicontinuous} if, for every $x\in P$, the family
\[
 \{\uparrow F:F\subseteq P\text{ finite and }F\ll x\}
\]
is filtered under inclusion and
\[
 \uparrow x=\bigcap_{F\ll x}\uparrow F.
\]
It is \emph{quasialgebraic} if, for every $x\in P$, the family
\[
 \{\uparrow F:F\subseteq P\text{ finite},\ x\in\uparrow F,\ F\ll F\}
\]
is filtered under inclusion and has intersection $\uparrow x$.
\end{definition}

For comparison, an element $c\in P$ is \emph{compact} if $c\ll c$.  A dcpo
is \emph{algebraic} if, for every $x\in P$, the compact elements below $x$
form a directed set with supremum $x$.
Every algebraic dcpo is quasialgebraic.

\section{A Counterexample from the Binary-Expansion Order}
\label{sec:construction}

Let $2=\{0,1\}$, let $2^n$ denote the set of binary words of length $n$,
and put
\[
T=2^{<\omega} =\bigcup_{n<\omega}2^n
\]
be the set of finite binary words, ordered by the prefix relation
$\preccurlyeq$.

Let
\[
X=2^\omega
\]
be the set of all infinite binary sequences.  If $\alpha\in X$, then
$\alpha\upharpoonright n\in2^n$ denotes its restriction to the first $n$
coordinates.  For $s\in T$ and $\alpha\in X$, write
\[
        s\preccurlyeq\alpha
        \quad\Longleftrightarrow\quad
        s=\alpha\upharpoonright |s|.
\]

Endow $2=\{0,1\}$ with the
discrete topology, and equip $X=2^{\omega}$ with the product topology.  Thus $X$ is
Cantor space.

For $s=(s_0,\ldots,s_{n-1})\in T$, where $n=|s|$, put
\[
[s]=\{\alpha\in X:\alpha\upharpoonright n=s\}
=\{\alpha\in X:s\preccurlyeq\alpha\}.
\]
Thus $[s]$ consists of all infinite binary sequences extending $s$.  The
cylinder $[s]$ is clopen: it is open because it prescribes only finitely
many coordinates, and
\[
X\setminus[s]=\bigcup_{\substack{t\in 2^n\\t\ne s}}[t]
\]
is open as well.  Moreover, the family $\{[s]:s\in T\}$ is a basis for the
product topology on $X$.

\begin{lemma}
	The Cantor space $X=2^\omega$ is a compact Hausdorff space.
\end{lemma}

\begin{proof}
	Since $2=\{0,1\}$ is a finite discrete space, it is compact.  Hence
	$X=2^\omega$ is compact by Tychonoff's theorem.  To see that $X$ is
	Hausdorff, let $\alpha,\beta\in X$ with $\alpha\ne\beta$.  Choose $n$ such
	that $\alpha(n)\ne\beta(n)$.  Then the two cylinders
	\[
	[\alpha\upharpoonright(n+1)]
	\qquad\text{and}\qquad
	[\beta\upharpoonright(n+1)]
	\]
	are disjoint open neighbourhoods of $\alpha$ and $\beta$, respectively.
	Thus $X$ is Hausdorff.
\end{proof}

Define the usual binary-value map
\[
\pi:X\longrightarrow\UnitInterval,\qquad
\pi(\alpha)=\sum_{n=0}^{\infty}\frac{\alpha(n)}{2^{n+1}}.
\]
The elementary properties of binary expansions give the following.
\begin{enumerate}
	\item The map $\pi$ is surjective but not injective.  For example,
	\[
	\pi(01111\ldots)=\pi(10000\ldots)=\frac12.
	\]
	For every $r\in\UnitInterval$, the fibre $\pi^{-1}(\{r\})$ is nonempty and
	has at most two elements; in particular, it is compact.
	
	\item If $s,t\in T$ and $s\preccurlyeq t$, then
	\[
	\pi([t])\subseteq\pi([s]).
	\]
\end{enumerate}
\begin{lemma}\label{lem:quotient}
	The binary-value map
	\[
	\pi:X\longrightarrow\UnitInterval
	\]
	is a quotient map.
\end{lemma}

\begin{proof}
	We first prove that $\pi$ is continuous.  Fix $\alpha\in X$, and let
	$\varepsilon>0$.  Choose $N\in\mathbb{N}$ such that
	\[
	\frac{1}{2^N}<\varepsilon.
	\]
	If $\beta\in[\alpha\upharpoonright N]$, then $\alpha(n)=\beta(n)$ for all
	$n<N$.  Hence
	\[
	\left|\pi(\alpha)-\pi(\beta)\right|
	\leq\sum_{n=N}^{\infty}\frac{1}{2^{n+1}}
	=\frac{1}{2^N}
	<\varepsilon.
	\]
	Thus $\pi$ is continuous.
	
	By the preceding discussion, $\pi$ is surjective.  Let $C\subseteq X$ be
	closed.  Since $X$ is compact, $C$ is compact.  Hence $\pi(C)$ is compact,
	and therefore closed in the Hausdorff space $\UnitInterval$.  Thus $\pi$ is
	a closed map.
	
	Finally, let $V\subseteq\UnitInterval$ be such that $\pi^{-1}(V)$ is open.
	Then $X\setminus\pi^{-1}(V)$ is closed, and hence
	\[
	\UnitInterval\setminus V
	=\pi\bigl(X\setminus\pi^{-1}(V)\bigr)
	\]
	is closed, where surjectivity of $\pi$ is used in the equality.  Therefore
	$V$ is open in $[0,1]$.  This proves that $\pi$ is a quotient map.
\end{proof}
We now define
\[
P=T\mathbin{\dot\cup}\UnitInterval
\]
to be the disjoint union of $T$ and $[0,1]$, equipped with the following order:
\begin{align*}
	s\leq t
	&\quad\Longleftrightarrow\quad s\preccurlyeq t
	&& (s,t\in T),\\
	s\leq r
	&\quad\Longleftrightarrow\quad r\in\pi([s])
	&& (s\in T,\ r\in\UnitInterval),\\
	r\leq q
	&\quad\Longleftrightarrow\quad r=q
	&& (r,q\in\UnitInterval).
\end{align*}
No other comparabilities are imposed.  Thus the points of $\UnitInterval$
are pairwise incomparable maximal elements.  A word $s\in T$ lies below
exactly those maximal elements having a binary expansion that extends $s$.

\begin{proposition}
	The relation $\leq$ is a partial order on $P$, and
	\[
	\operatorname{Max}(P)=\UnitInterval.
	\]
\end{proposition}

\begin{proof}
	Reflexivity and antisymmetry are immediate.  For transitivity, the only
	nontrivial mixed case is
	\[
	s\preccurlyeq t\leq r,
	\qquad s,t\in T,\quad r\in\UnitInterval.
	\]
	Since $t\leq r$, we have $r\in\pi([t])$.  As
	\[
	\pi([t])\subseteq\pi([s]),
	\]
	it follows that $r\in\pi([s])$, and hence $s\leq r$.  All remaining cases
	follow from transitivity of the prefix order or from equality on
	$\UnitInterval$.
	
	Moreover, by the definition of the order,
$
	\operatorname{Max}(P)=\UnitInterval$.
\end{proof}

\begin{proposition}\label{prop:well-founded}
	The poset $P$ is well-founded.
\end{proposition}

\begin{proof}
	In every strictly descending chain, at most the first element can lie in
	$\UnitInterval$.  All remaining elements lie in $T$, where strict descent
	in the prefix order strictly decreases word length.  Hence every strictly
	descending chain is finite.
\end{proof}

\begin{remark}
	The one-letter word $(1)\in T$ must not be confused with the real number
	$1\in\UnitInterval$.  They are distinct elements of the disjoint union
	\[
	P=T\mathbin{\dot\cup}\UnitInterval.
	\]
	Since
	\[
	\pi([(1)])=\left[\frac12,1\right],
	\]
	we have
	\[
	(1)\leq r
	\quad\Longleftrightarrow\quad
	r\in\left[\frac12,1\right]
	\qquad(r\in\UnitInterval).
	\]
	In particular,
	\[
	(1)\leq 1,
	\qquad\text{but}\qquad
	(1)\neq 1.
	\]
	Moreover, $1\nleq(1)$, since $1$ is a maximal element of $P$.
\end{remark}

\subsection{Directed subsets and the dcpo property}

\begin{lemma}[Classification of directed sets]\label{lem:directed}
	Let $D$ be a nonempty directed subset of $P$.
	\begin{enumerate}
	\item If $D$ contains a maximal element $r\in\UnitInterval$, then $r$ is the maximum of $D$,
		and hence $\bigvee D=r$.
		\item If $D\subseteq T$, then $D$ is a prefix chain.  If its word lengths are
		bounded, it has a maximum.  If its word lengths are unbounded, then there is a
		unique $\alpha\in X$ such that
		\[
		\bigvee D=\pi(\alpha),\qquad
		\alpha=\bigcup_{s\in D}s.
		\]
	\end{enumerate}
\end{lemma}

\begin{proof}
	Suppose first that $r\in D\cap\UnitInterval$.  For every $d\in D$, directedness
	provides $u\in D$ with $r,d\le u$.  Since $r$ is maximal, $u=r$, so $d\le r$.
	Thus $r$ is the maximum of $D$.
	
	Assume next that $D\subseteq T$.  Two incomparable finite words have no common upper
	bound in $T$, so directedness forces every two elements of $D$ to be comparable.
	Therefore $D$ is a prefix chain.  If the lengths are bounded, only finitely many
	words can occur, and a word of maximum length is the maximum of $D$.
	
	Finally, suppose that the lengths are unbounded and put $\alpha=\bigcup D$.  Then
	every $s\in D$ is a prefix of $\alpha$, and hence $s\le\pi(\alpha)$.  So
	$\pi(\alpha)$ is an upper bound.  Let $r$ be any upper bound of $D$.  It cannot be a
	finite word, because then its length would bound the lengths in $D$.  Hence
	$r\in\UnitInterval$.  For every $n$, some element of $D$ extends
	$\alpha\upharpoonright n$.  Since that element is below $r$,
	\[
	\pi^{-1}(\{r\})\cap[\alpha\upharpoonright n]\ne\varnothing.
	\]
	The fibre $\pi^{-1}(\{r\})$ is compact, and its intersections with the
	nested cylinders $[\alpha\upharpoonright n]$ are nonempty.  Their
	intersection is therefore nonempty; since
	\[
	\bigcap_{n\in\Nat}[\alpha\upharpoonright n]=\{\alpha\},
	\]
	we have $\alpha\in\pi^{-1}(\{r\})$, and therefore $r=\pi(\alpha)$.  This proves both uniqueness and
	the asserted supremum.
\end{proof}

\begin{proposition}\label{prop:dcpo}
	$P$ is a dcpo.
\end{proposition}

\begin{proof}
	Lemma~\ref{lem:directed} gives a supremum for every nonempty directed subset of $P$.
\end{proof}

\subsection{An exact finite-approximation criterion}

For $s\in T$, put
\[
        C_s=\pi^{-1}(\pi([s])).
\]
The compact set $C_s$ contains all branches whose binary values fall in the
interval $\pi([s])$, not merely the branches in $[s]$.  This enlargement is
needed because a dyadic value may have two binary expansions.  For example,
let $s=(1)$.  Then
\[
\pi([s])=\{\pi(\alpha):s\preccurlyeq \alpha\in 2^\omega\}=\left[\frac12,1\right].
\]
Although $(01111\cdots)\notin[s]$, we have
\[
\pi(01111\cdots)=\frac12=\pi(10000\cdots)\in\pi([s]).
\]
Hence
\[
(01111\cdots)\in C_s\setminus[s].
\]
In fact,
\[
C_{(1)}=[(1)]\cup\{(01111\cdots)\}.
\]

Let $E\subseteq T=2^{<\omega}$ be a finite set. 
Put
\[
U_E=\bigcup_{e\in E}[e]\subseteq 2^\omega.
\]
This is a finite union of cylinders, hence clopen. 

\begin{lemma}\label{lem:criterion}
	Let $r\in\UnitInterval$, $s\in T$, and let $F\subseteq P$ be finite.
	Then
	\begin{itemize}
		\item [\rm (1)] $F\ll r
		\quad\Longleftrightarrow\quad
		\pi^{-1}(\{r\})\subseteq U_{F\cap 2^{<\omega}}$. 
		\item [\rm (2)] $F\ll s
		\quad\Longleftrightarrow\quad
		s\in \UpSet (F\cap 2^{<\omega})
		\quad\text{and}\quad
		\pi^{-1}(\pi([s]))\subseteq U_{F\cap 2^{<\omega}}$. 
	\end{itemize}
	In particular, the maximal elements $F\cap[0,1]$ do not affect either
	condition.  Consequently, for every $x\in P$,
	\[
	F\ll x \quad\Longleftrightarrow\quad F\cap T\ll x.
	\]
\end{lemma}

\begin{proof}
	
	(1)  Suppose $F\ll r$, and let
	$\alpha\in \pi^{-1}(\{r\})$.  Consider the directed prefix chain
	\[
	D_\alpha=\{\alpha\upharpoonright n:n\in\Nat\}.
	\]
	Its supremum is $r$ by Lemma~\ref{lem:directed}.  Hence
	$D_\alpha\cap\UpSet F\ne\varnothing$, which means some member
	of $D_{\alpha}$ lies above a word $e\in F\cap 2^{<\omega}$.  Equivalently, $e\preccurlyeq\alpha$, and
	$\alpha\in U_{F\cap 2^{<\omega}}$.  This proves $\pi^{-1}(\{r\})\subseteq U_{F\cap 2^{<\omega}}$.
	
	Conversely, assume $\pi^{-1}(\{r\})\subseteq U_{F\cap 2^{<\omega}}$, and let $D$ be directed with
	$\bigvee D\in\UpSet r$.  Since $r$ is maximal, $\bigvee D=r$.  
	We consider two cases:
	\begin{itemize}
		\item [(i)] If $r\in D$, choose
		$\alpha\in\pi^{-1}(\{r\})$.  The assumption gives
		$\alpha\in U_{F\cap T}$, so there is $e\in F\cap T$ with
		$e\preccurlyeq\alpha$.  Then $e\le r$, and hence
		$r\in\UpSet F$.
		\item [(ii)] 
		If $r\notin D$, the classification lemma says that $D$ is an
		unbounded prefix chain whose union is some $\alpha\in \pi^{-1}(\{r\})\subseteq U_{F\cap 2^{<\omega}}$.  Since
		$\alpha\in U_{F\cap 2^{<\omega}}$, there exists some $e\in F\cap 2^{<\omega}$ that is a prefix of $\alpha$, and a sufficiently long
		element of $D$ lies above $e$.
	\end{itemize}
	In either case $D\cap\UpSet F\ne\varnothing$, hence
	$F\ll r$.
	
	\medskip
	(2)  If $F\ll s$, applying the definition to the
	singleton directed set $\{s\}$ gives $s\in\UpSet F$. Then there is $e\in F\cap 2^{<\omega}$ with $e\preccurlyeq s$.
	This shows that $s\in\UpSet(F\cap T)$.
	Next, let $\alpha\in \pi^{-1}(\pi([s]))$ and put
	$r=\pi(\alpha)$.  Then $s\le r$, so the prefix chain $D_\alpha = \{\alpha\upharpoonright n:n\in\mathbb N\}$ has supremum in
	$\UpSet s$.  It must meet $\UpSet F$, and again this can happen only above a word in
	$F\cap 2^{<\omega}$.  Thus $\alpha\in U_{F\cap 2^{<\omega}}$. 
	
	For the converse, assume the two conditions and take a directed subset $D\subseteq P$ with
	$\bigvee D\in\UpSet s$.  There are three cases.
	\begin{itemize}
		\item [(i)] If $\bigvee D=t\in T$, then $s\preccurlyeq t$.  The classification lemma gives
		$t\in D$ as a maximum.  Choose $e\in F\cap 2^{<\omega}$ with $e\preccurlyeq s$.  Then $e\le t$, hence
		$t\in\UpSet F$.
		\item [(ii)] If $\bigvee D=r\in\UnitInterval$ and $r\in D$, then from $s\leq \bigvee D$ it holds that
		$r\in\pi([s])$, so $\pi^{-1}(\{r\})\subseteq \pi^{-1}(\pi([s]))\subseteq U_{F\cap 2^{<\omega}}$.  For
		$\alpha\in \pi^{-1}(\{r\})$, choose $e\in F\cap 2^{<\omega}$ with $e\preccurlyeq\alpha$.  Then 
		$r=\pi(\alpha)\in\pi([e])$, and hence $e\le r$ by the
		definition of the order on $P$.
		\item [(iii)] If $\bigvee D=r\in\UnitInterval$ but $r\notin D$, then $D$ is an
		unbounded prefix chain.  Let $\alpha$ be its union.  We have
		$\pi(\alpha)=r\in\pi([s])$, so $\alpha\in \pi^{-1}(\pi([s]))\subseteq U_{F\cap 2^{<\omega}}$.  Hence some
		$e\in F\cap 2^{<\omega}$ is a prefix of $\alpha$, and a sufficiently long member of $D$ lies
		above $e$.
	\end{itemize}
	Every case gives $D\cap\UpSet F\ne\varnothing$.  Therefore $F\ll s$.
\end{proof}

\subsection{Quasicontinuity}

\begin{lemma}\label{lem:refine}
	Let $A$ be a compact set in the Cantor space $X=2^{\omega}$, and suppose
	$A\subseteq U_{E_1}\cap U_{E_2}$ for finite subsets $E_1,E_2\subseteq T$.  Then there is a
	finite $E_0\subseteq T$ such that
	\[
	A\subseteq U_{E_0}\quad \text{ and }\quad E_0\subseteq \UpSet E_1\cap\UpSet E_2.
	\]
\end{lemma}

\begin{proof}
	For each $\alpha\in A$, choose $e_i\in E_i$ with $e_i\preccurlyeq\alpha$ for $i\in\{1,2\}$.
	The two chosen prefixes are comparable.  Take a sufficiently long prefix
	$e_\alpha$ of $\alpha$ extending both of them.  The cylinders $[e_\alpha]$ cover
	$A$, and compactness gives a finite subcover.  Take $E_0$ to be the corresponding
	finite set of words.
\end{proof}

\begin{proposition}[Filteredness]\label{prop:filtered}
	For every $x\in P$, the family
	\[
	\{\UpSet F:F\ll x\}
	\]
	is filtered under inclusion.
\end{proposition}

\begin{proof}
	Suppose $F_1,F_2\ll x$, and put
	\[
	E_i=F_i\cap T\qquad(i=1,2).
	\]
	Define
	\[
	C_x=
	\begin{cases}
		\pi^{-1}(\{r\}),&x=r\in\UnitInterval,\\
		\pi^{-1}(\pi([s])),&x=s\in T.
	\end{cases}
	\]
	Every $C_x$ is a nonempty compact subset of $X$.
	By Lemma~\ref{lem:criterion},
	\[
	C_x\subseteq U_{E_1}\cap U_{E_2}.
	\]
	Apply Lemma~\ref{lem:refine} to obtain a finite subset $E_0\subseteq T$
	such that
	\[
	C_x\subseteq U_{E_0}
	\qquad\text{and}\qquad
	E_0\subseteq\UpSet E_1\cap\UpSet E_2.
	\]
	
	We consider two cases:
	\begin{itemize}
		\item [(i)] If $x=r\in\UnitInterval$, take $E=E_0$.  Since
		$\pi^{-1}(\{r\})=C_x\subseteq U_E$, Lemma~\ref{lem:criterion}
		gives $E\ll r$.  In addition, since each $e\in E_0$ lies above a word of $E_1$
		and above a word of $E_2$, it follows that
		\[
		\UpSet E\subseteq\UpSet F_1\cap\UpSet F_2.
		\]
		\item [(ii)] If $x=s$ is a finite word, let $E=E_0\cup\{s\}$.  The set $E_0$ covers $\pi^{-1}(\pi([s]))=C_s$ and
		$s\preccurlyeq s$, so $E\ll s$ by Lemma~\ref{lem:criterion}.  In addition,
		$F_i\ll s$ implies that a word of $E_i$ is a prefix of $s$.  Hence
		$s\in\UpSet F_i$, and the same containment
		\[
		\UpSet E\subseteq\UpSet F_1\cap\UpSet F_2
		\]
		holds.
	\end{itemize}
	In either case, $E\subseteq P$ is finite, $E\ll x$, and
	$\UpSet E\subseteq\UpSet F_1\cap\UpSet F_2$. 
\end{proof}

We next establish the intersection condition.  We use two elementary cylinder
separation facts.

\begin{lemma} \label{lem:separate}
	\begin{itemize}
		\item [\rm (1)] If $A,B\subseteq X$ are disjoint compact sets, then there is a finite subset
		$E\subseteq T$ such that
		\[
		A\subseteq U_E\qquad\text{and}\qquad U_E\cap B=\varnothing.
		\]
		\item [\rm (2)] If $A\subseteq X$ is compact and $t\in T$, then there is a finite
		$E_t\subseteq T$ such that
		\[
		A\subseteq U_{E_t}
		\qquad\text{and}\qquad t\notin \UpSet E_t.
		\]
	\end{itemize}
\end{lemma}

\begin{proof}
	(1) Since the compact set $B$ is closed in the Hausdorff space $X$,  
	for every $\alpha\in A$, the open set $X\setminus B$ contains a cylinder
	$[s_\alpha]$ such that
	\[
	\alpha\in[s_\alpha]\subseteq X\setminus B.
	\]
	The family $\{[s_\alpha]:\alpha\in A\}$ of open sets covers $A$.  By compactness, choose
	$\alpha_1,\ldots,\alpha_m\in A$ such that
	\[
	A\subseteq\bigcup_{j=1}^m[s_{\alpha_j}].
	\]
	Then $E=\{s_{\alpha_1},\ldots,s_{\alpha_m}\}$ satisfies
	\[
	A\subseteq U_E\qquad\text{and}\qquad U_E\cap B=\varnothing.
	\]
	
	(2) If $A=\varnothing$, take $E_t=\varnothing$ and we are done.  Hence assume that
	$A\ne\varnothing$.  For each
	$\alpha\in A$, choose $N_\alpha>|t|$ and put
	$s_\alpha=\alpha\upharpoonright N_\alpha$.  The cylinders $[s_\alpha]$
	cover $A$, so compactness yields $\alpha_1,\ldots,\alpha_m\in A$ such that
	\[
	A\subseteq\bigcup_{j=1}^m[s_{\alpha_j}].
	\]
	Let $E_t=\{s_{\alpha_1},\ldots,s_{\alpha_m}\}$.  Then
	$A\subseteq U_{E_t}$.  Since every $e\in E_t$ has length strictly greater
	than $|t|$, no such $e$ can be a prefix of $t$.  Therefore
	$t\notin\UpSet E_t$.
\end{proof}

\begin{theorem}\label{thm:quasicontinuous}
	$P$ is a quasicontinuous dcpo.
\end{theorem}

\begin{proof}
	By Proposition~\ref{prop:filtered}, it remains to prove that, for every
	$x\in P$,
	\[
	\UpSet x=\bigcap_{F\ll x}\UpSet F.
	\]
	The inclusion
	\[
	\UpSet x\subseteq\bigcap_{F\ll x}\UpSet F
	\]
	is immediate.
	For the converse, let $y\notin\UpSet x$.  We construct a finite set $E\subseteq T$
	such that $E\ll x$ and $y\notin\UpSet E$.
	
	\medskip
	\noindent\textbf{Case 1: $x=r\in\UnitInterval$.}
	\begin{itemize}
		\item [(i)] If $y=q\in\UnitInterval$, then $q\ne r$ (because $y\notin \UpSet x$), so $\pi^{-1}(\{r\})$ and $\pi^{-1}(\{q\})$ are disjoint compact
		sets.  By Lemma~\ref{lem:separate}(1), choose finite $E\subseteq T$ with
		$\pi^{-1}(\{r\})\subseteq U_E$ and $U_E\cap \pi^{-1}(\{q\})=\varnothing$.  The first property gives
		$E\ll r$ by Lemma~\ref{lem:criterion}, while the second says that no $e\in E$ lies below $q$.  Thus
		$q\notin\UpSet E$.
		
		\item [(ii)] If $y=t\in T$, apply Lemma~\ref{lem:separate}(2) with
		$A=\pi^{-1}(\{r\})$.  We obtain a finite set $E=E_t\subseteq T$
		with $\pi^{-1}(\{r\})\subseteq U_E$ and $t\notin\UpSet E$.
		Lemma~\ref{lem:criterion} gives $E\ll r$.
	\end{itemize}

	\medskip
	\noindent\textbf{Case 2: $x=s\in T$.}
	\begin{itemize}
		\item [(i)] If $y=q\in\UnitInterval$, then $q\notin\UpSet s$ means
		$q\notin\pi([s])$.  Hence $\pi^{-1}(\pi([s]))\cap \pi^{-1}(\{q\})=\varnothing$.  By Lemma~\ref{lem:separate}(1), there exists a finite $E_0\subseteq T$ such that  $\pi^{-1}(\pi([s]))\subseteq U_{E_0}$ and 
		$U_{E_0}\cap \pi^{-1}(\{q\})=\varnothing$.  Let $E=E_0\cup\{s\}$.  Since
		$q\notin\pi([s])$, we also have $[s]\cap \pi^{-1}(\{q\})=\varnothing$.  Therefore
		$q\notin\UpSet E$.  On the other hand, $\pi^{-1}(\pi([s]))\subseteq U_E$ and $s\in E$, so by
		Lemma~\ref{lem:criterion}, we have $E\ll s$.
		
		\item [(ii)] If $y=t\in T$, then $t\notin\UpSet s$ means
		$s\not\preccurlyeq t$.  Put $A=\pi^{-1}(\pi([s]))$.  By
		Lemma~\ref{lem:separate}(2), there is a finite set $E_t\subseteq T$
		such that
		\[
		A\subseteq U_{E_t}
		\qquad\text{and}\qquad t\notin \UpSet E_t,
		\]
		and let $E=E_t\cup\{s\}$.  No member of $E_t$ is a prefix of $t$, and $s$ is not a prefix of
		$t$.  Thus $t\notin\UpSet E$. By Lemma~\ref{lem:criterion}, 
		$E\ll s$.
	\end{itemize}
	
	In all cases, an element outside $\UpSet x$ is excluded by one $\UpSet E$ with
	$E\ll x$.  This proves the reverse inclusion.
\end{proof}

\subsection{Failure of quasialgebraicity}

\begin{lemma}\label{lem:saturation}
	Let $F\subseteq P$ be a finite set and $E=F\cap T$.  If $F\ll F$ and
	$\UpSet F\ne\varnothing$, then
	\[
	E\ne\varnothing
	\qquad\text{and}\qquad
	\pi^{-1}(\pi(U_E))=U_E.
	\]
	Therefore, $\pi(U_E)$ is an open set in $[0,1]$.
\end{lemma}

\begin{proof}
	Suppose first that $E=F\cap T=\varnothing$.  Then $F\subseteq\UnitInterval$.
	Since $\UpSet F\ne\varnothing$, choose $m\in F$.  For any
	$\alpha\in\pi^{-1}(\{m\})$, the prefix chain
	$
	D_\alpha=\{\alpha\upharpoonright n:n\in\mathbb{N}\}
	$
	has supremum $m\in\UpSet F$.  On the other hand, $D_\alpha\subseteq T$,
	whereas every element of $F$ belongs to $\UnitInterval$.  Hence
	\[
	D_\alpha\cap\UpSet F=\varnothing,
	\]
	contradicting $F\ll F$.  Therefore $E\ne\varnothing$.
	
	We next prove the following claim:
	\[
	\alpha\in U_E\ \text{and}\ \pi(\alpha)=\pi(\beta)
	\quad\Longrightarrow\quad
	\beta\in U_E.
	\]
	Suppose, to the contrary, that $\beta\notin U_E$.  Since $\alpha\in U_E$,
	there is $e\in E$ such that $e\preccurlyeq\alpha$.  Put
	\[
	r=\pi(\alpha)=\pi(\beta).
	\]
	Then $e\leq r$, so $r\in\UpSet F$.  The prefix chain
	$
	D_\beta=\{\beta\upharpoonright n:n\in\mathbb{N}\}
	$
	has supremum $r$.  Since $\beta\notin U_E$, no member of $D_\beta$ lies
	above an element of $E$.  Moreover, $D_\beta\subseteq T$, so no member of
	$D_\beta$ lies above an element of $F\cap\UnitInterval$.  Consequently,
	$
	D_\beta\cap\UpSet F=\varnothing,
	$
	again contradicting $F\ll F$.  This proves the claim.
	
	Now let $\beta\in\pi^{-1}(\pi(U_E))$.  Then there is some
	$\alpha\in U_E$ such that
	$
	\pi(\alpha)=\pi(\beta).
	$
	The claim gives $\beta\in U_E$.  Therefore
	$
	\pi^{-1}(\pi(U_E))\subseteq U_E.
	$
	The reverse inclusion is immediate, and hence
	$
	\pi^{-1}(\pi(U_E))=U_E.
	$

	Since $U_E$ is open in $X$ and $\pi$ is a quotient map
	(Lemma~\ref{lem:quotient}), it follows that $\pi(U_E)$ is open in
	$\UnitInterval$.
\end{proof}
\begin{proposition}
	\label{prop:allleaves}
	If $F\ll F$ and $\UpSet F\ne\varnothing$, then
	$[0,1]\subseteq\UpSet F.
	$
\end{proposition}

\begin{proof}
	By Lemma~\ref{lem:saturation}, $U_E$ is a nonempty clopen $\pi$-saturated subset of
	$X$.  The continuous surjection
	\[
	\pi:X\twoheadrightarrow[0,1]
	\]
	is a quotient map (Lemma~\ref{lem:quotient}).  By Lemma~\ref{lem:saturation}, 
	\[
	\pi^{-1}(\pi(U_E))=U_E.
	\]
	As $U_E$ is open, the quotient-map property implies that $\pi(U_E)$ is open in
	$[0,1]$.  It is also compact and hence closed.  Connectedness of $[0,1]$ now gives
	\[
	\pi(U_E)=[0,1].
	\]
	Thus, $U_E=\pi^{-1}(\pi(U_E))=\pi^{-1}([0,1])=X=2^{\omega}$.
	
	Let $r\in[0,1]$ and choose $\alpha\in \pi^{-1}(\{r\})$.  Since $\alpha\in U_E=X$, some
	$e\in E$ is a prefix of $\alpha$.  Thus $e\le r$, so $r\in\UpSet F$.
\end{proof}

\begin{theorem}\label{thm:notquasialgebraic}
	$P$ is not quasialgebraic.
\end{theorem}

\begin{proof}
	Fix two distinct leaves, say $r=0$ and $q=1$.  If $F\ll F$ and
	$r\in\UpSet F$, then $\UpSet F\ne\varnothing$, so
	Proposition~\ref{prop:allleaves} implies $q\in\UpSet F$.  Hence
	\[
	q\in\bigcap\{\UpSet F:r\in\UpSet F,\ F\ll F\}.
	\]
	But $q\notin\UpSet r=\{r\}$.  Thus the intersection clause in the definition of
	quasialgebraicity fails at $r$.
\end{proof}

\subsection{Negative answers to the two problems}

\begin{lemma}\label{lem:chains}
	Every chain in $P$ is order-isomorphic to a suborder of $\omega+1$.
	In particular, every chain in $P$ is countable.
\end{lemma}

\begin{proof}
	Let $C$ be a chain in $P$.  Since distinct elements of $\UnitInterval$ are
	incomparable, $C$ contains at most one element of $\UnitInterval$.  The
	set $C\cap T$ is a prefix chain, and the length map
	\[
	        s\longmapsto |s|
	\]
	is a strict order embedding of $C\cap T$ into $\omega$.  If $C$ contains a
	point of $\UnitInterval$, that point is the largest element of $C$.
	Consequently, $C$ is order-isomorphic to a suborder of $\omega+1$.
\end{proof}

\begin{theorem}[Negative answer to Problem 1.15]
	$P$ contains no sub-dcpo isomorphic to $[0,1]$.
\end{theorem}

\begin{proof}
	In fact, $P$ contains no subposet isomorphic to $[0,1]$.  Such a subposet would be an
	uncountable chain, contradicting Lemma~\ref{lem:chains}.
\end{proof}

\begin{theorem}[Negative answer to Problem 1.14]
	The unit interval $[0,1]$ is not a Scott-continuous retract of $P$.
\end{theorem}

\begin{proof}
	Suppose that Scott-continuous maps
	\[
	f:P\longrightarrow[0,1],
	\qquad
	g:[0,1]\longrightarrow P
	\]
	satisfy $f\circ g=\operatorname{id}_{[0,1]}$.  Then $g$ is injective.  Scott
	continuity implies monotonicity, so $g([0,1])$ is an uncountable chain in $P$.
	This contradicts Lemma~\ref{lem:chains}.
\end{proof}

\begin{corollary}
	Problems~1.14 and~1.15 both have negative answers.
\end{corollary}

\begin{proof}
	The poset $P$ is quasicontinuous and nonquasialgebraic, while the two preceding
	theorems rule out, respectively, a sub-dcpo isomorphic to $[0,1]$ and a
	Scott-continuous retract.
\end{proof}

\section{An interval criterion for arbitrary dcpos}\label{sec:criterion}

For completeness, the example also has a transparent Scott-continuous surjection
onto $[0,1]$.  For $s\in T$, define
\[
a(s)=\min\pi([s]),
\]
the left endpoint of the binary interval determined by $s$.  Let
\[
h:P\longrightarrow[0,1],\qquad
h(r)=r\quad(r\in[0,1]),\qquad
h(s)=a(s)\quad(s\in T).
\]

\begin{proposition}
	The map $h$ is a Scott-continuous surjection.
\end{proposition}

\begin{proof}
	Surjectivity is immediate because $h$ is the identity on the maximal part
	$[0,1]$.
	If $s\preccurlyeq t$, then $\pi([t])\subseteq\pi([s])$, hence $a(s)\le a(t)$.  If
	$s\le r$, then $a(s)\le r$.  Thus $h$ is monotone.
	
	By Lemma~\ref{lem:directed}, it remains only to check an unbounded prefix chain.
	If $D\subseteq T$ is such a chain and $\alpha=\bigcup D$, then
	\[
	\sup_{s\in D}h(s)
	=\sup_{s\in D}a(s)
	=\pi(\alpha)
	=h\!\left(\bigvee D\right).
	\]
	The middle equality follows because the left endpoints of the nested binary
	intervals converge to the value of the branch.  Every other directed set has a
	maximum, so preservation of its supremum follows from monotonicity.  Hence $h$
	preserves directed suprema and is Scott continuous.
\end{proof}

We conclude with a general criterion which explains the equivalence used in
the introduction.

\begin{lemma}[A rational chain yields an interval retract]\label{lem:qchain}
Let $P$ be a dcpo containing a chain
\[
 \{c_q:q\in\Q\}
\]
order-isomorphic to $\Q$.  Then there are Scott-continuous maps
\[
 e:\Unit\longrightarrow P,\qquad
 \rho:P\longrightarrow\Unit
\]
such that $\rho\circ e=\operatorname{id}_{\Unit}$.  Moreover, $e(\Unit)$ is
a sub-dcpo Scott-isomorphic to $\Unit$.
\end{lemma}

\begin{proof}
For $t\in\Unit$, define
\begin{equation}\label{eq:e}
 e(t)=\bigvee\{c_q:q<t\}.
\end{equation}
The set on the right is nonempty and directed.  If $s<t$, choose rationals
$s<q<r<t$.  Then
\[
 e(s)\leq c_q<c_r\leq e(t),
\]
so $e$ is a strict order embedding.

Let $D\subseteq\Unit$ be directed and put $t=\bigvee D$.  The set $e[D]$ is
directed.  If $q<t$, some $d\in D$ satisfies $q<d$, since otherwise $q$
would be an upper bound of $D$.  Hence
\[
 c_q\leq e(d)\leq\bigvee e[D].
\]
Taking suprema over $q<t$ gives $e(t)\leq\bigvee e[D]$; the reverse
inequality follows from monotonicity.  Therefore
\begin{equation}\label{eq:edirected}
 e\!\left(\bigvee D\right)=\bigvee e[D].
\end{equation}
Thus $e$ is Scott-continuous.  If $A\subseteq e(\Unit)$ is directed, then
$e^{-1}[A]$ is directed because $e$ reflects order, and
\eqref{eq:edirected} gives
\[
 \bigvee_P A=e\!\left(\bigvee_{\Unit}e^{-1}[A]\right)\in e(\Unit).
\]
Hence the image is a sub-dcpo.

For $p\in P$, put
\begin{equation}\label{eq:rho}
 \rho(p)=\sup\left(\{0\}\cup
       \{t\in\Unit:p\not\leq e(t)\}\right).
\end{equation}
If $p\leq p'$, then $p\not\leq e(t)$ implies $p'\not\leq e(t)$, so $\rho$
is monotone.  For $t\in\Unit$, set
\[
V_t=P\setminus\downarrow e(t)
\]
which is Scott-open, since $\downarrow e(t)$ is Scott-closed.  For
$0\leq a<1$,
\begin{equation}\label{eq:rhopreimage}
 \rho^{-1}((a,1])=\bigcup_{\substack{t\in\Unit\\t>a}}V_t.
\end{equation}
Since
\[
 \{\Unit\}\cup\{(a,1]:0\leq a<1\}
\]
is a basis for the Scott topology of $\Unit$, $\rho$ is Scott-continuous.

For completeness, let $D\subseteq P$ be directed, $p=\bigvee D$, and
$u=\bigvee\rho[D]$.  Monotonicity gives $u\leq\rho(p)$.  If
$u<\rho(p)$, choose $a$ with $u<a<\rho(p)$.  The Scott-open set in
\eqref{eq:rhopreimage} contains $p$, hence meets $D$, giving $d\in D$ with
$\rho(d)>a>u$, a contradiction.  Therefore
\[
 \rho\!\left(\bigvee D\right)=\bigvee\rho[D].
\]
Finally, the order-embedding property of $e$ gives
\[
 e(s)\not\leq e(t)\quad\Longleftrightarrow\quad t<s.
\]
Substitution into \eqref{eq:rho} yields
\[
 \rho(e(s))=\sup\bigl(\{0\}\cup\{t\in\Unit:t<s\}\bigr)=s.
\]
\end{proof}

\begin{proposition}\label{prop:equivalence}
For every dcpo $P$, the following are equivalent.
\begin{enumerate}
\item $P$ contains a chain order-isomorphic to $\Q$.
\item $P$ contains a sub-dcpo Scott-isomorphic to $\Unit$.
\item $\Unit$ is a Scott-continuous retract of $P$.
\end{enumerate}
\end{proposition}

\begin{proof}
Lemma~\ref{lem:qchain} proves that (1) implies (2) and (3).  Clearly (2)
implies (1).  If $\rho\circ e=\operatorname{id}_{\Unit}$, then $e$ is an
order embedding.  Moreover, if $A\subseteq e(\Unit)$ is directed, then
$e^{-1}[A]$ is directed because $e$ reflects order, and Scott continuity gives
\[
 \bigvee_P A=e\!\left(\bigvee_{\Unit}e^{-1}[A]\right)\in e(\Unit).
\]
Thus $e(\Unit)$ is a sub-dcpo Scott-isomorphic to $\Unit$, proving (3)
implies (2).
\end{proof}

\end{document}